\documentclass[11pt]{amsart}
\usepackage[margin=3cm]{geometry}
\usepackage[utf8]{inputenc}
\usepackage{mathtools,amsfonts,amsthm,
mathrsfs,amssymb,amsmath}
\usepackage{textcomp}
\mathtoolsset{showonlyrefs}
\usepackage{bookmark}
\usepackage{enumerate}
\usepackage{microtype}

\usepackage{todonotes}

\definecolor{DarkDesaturatedBlue}{HTML}{3A3556}
\definecolor{VividOrange}{HTML}{F15918}
\definecolor{PureOrange}{HTML}{FFBA00}
\definecolor{LightGrayishPink}{HTML}{EEC5D5}
\definecolor{VerySoftBlue}{HTML}{B5AFDB}

\newtheorem{thm}{Theorem}
\newtheorem{lemma}{Lemma}

\newtheorem{corollary}{Corollary}

\theoremstyle{definition}
\newtheorem*{definition*}{Definition}

\newcommand{\N}{\mathcal{N}}

\newcommand{\C}{\mathscr{C}}

\newcommand{\bc}{\mathbf{c}}
\newcommand{\bk}{\mathbf{k}}

\newcommand{\eps}{\varepsilon}

\newcommand{\pr}[1]{\mathbb{P}\left[#1\right]}
\newcommand{\esp}[1]{\mathbb{E}\left[#1\right]}

\newcommand{\pth}[1]{\left(#1\right )}

\newcommand{\restrict}[2]{{#1}_{|{#2}}}

\title{Colouring locally sparse graphs with the first moment method}

\author[E.\ Hurley]{Eoin Hurley}
\address{Universität Heidelberg, Germany.}
\email{hurley@informatik.uni-heidelberg.de}
\thanks{The research leading to these results was partially supported by the  Deutsche Forschungsgemeinschaft (DFG, German Research Foundation) -- 428212407 (E. Hurley)}
\author[F.\ Pirot]{Fran\c{c}ois Pirot}
\address{Université Paris-Saclay, France.}
\email{francois.pirot@lisn.fr}

\begin{document}
\begin{abstract}
We give a short proof of a bound on the list chromatic number of graphs $G$ of maximum degree $\Delta$ where each neighbourhood has density at most $d$, namely $\chi_\ell(G) \le (1+o(1)) \frac{\Delta}{\ln \frac{\Delta}{d+1}}$ as $\frac{\Delta}{d+1} \to \infty$. This bound is tight up to an asymptotic factor $2$, which is the best possible barring a breakthrough in Ramsey theory, and strengthens results due to Vu, and more recently Davies, P., Kang, and Sereni. Our proof relies on the first moment method, and adapts a clever counting argument developed by Rosenfeld in the context of non-repetitive colourings. 
As a final touch, we show that our method provides an asymptotically tight lower bound on the number of colourings of locally sparse graphs.

\end{abstract}

\maketitle

\section{Introduction}
Given a  graph $G$, 
a \emph{list assignment} of $G$ is a map $L\colon V(G) \rightarrow 2^{\mathbb{N}}$, and a \emph{proper $L$-colouring} $c$ of $G$ is a map $c\colon V(G)\rightarrow \mathbb{N}$ such that $c(v)\in L(v)$ for every vertex $v \in V(G)$, and $c(u)\neq c(v)$ for every edge $uv\in E(G)$. If there exists an $L$-colouring of $G$, we say that $G$ is \emph{$L$-colourable}.
The \emph{list chromatic number} of $G$, denoted $\chi_\ell(G)$, is the minimum $k$ such that $G$ is $L$-colourable for every list assignment $L$ with $|L(v)|\ge k$ for every vertex $v\in V(G)$.

In a seminal result, Johansson \cite{Joh96} showed that the list chromatic number of triangle-free graphs of maximum degree $\Delta$ never exceeds $O(\Delta/\ln \Delta)$ as $\Delta \to \infty$. Two decades later, Molloy \cite{Mol19} showed with the help of entropy compression that this bound can be tightened to $(1+o(1))\Delta/\ln \Delta$ as $\Delta\to\infty$;  Bernshteyn \cite{Ber19} then showed that it was possible to replace the use of entropy compression with an application of the lopsided Lov\' asz Local Lemma and obtain a similar result. Following that breakthrough, there has been a lot of interest for triangle-free graph colourings and extensions \cite{AIS19, ABD21+, BKNP18+, DJKP20, DKPS20+}.

Although quite different, all the proofs in those works rely on a $2$-step colouring procedure. 
The first step is the one that varies the most among those proofs; given a graph $G$ and a list assignment $L$, it consists in constructing a (pseudo-)random partial $L$-colouring $c_0$ of $G$. In the work of Johansson \cite{Joh96}, that partial colouring was obtained with an iterative colouring procedure and analysed with a nibble argument, while more recent proofs rely on a more straightforward construction of that partial colouring; for instance Bernshteyn picks a partial colouring uniformly at random in his proof \cite{Ber19}.
After that first step, the set $V_1$ of uncoloured vertices induces a new list assignment $L_1\colon V_1 \to 2^\mathbb{N}$. After showing that with positive probability, the maximum colour degree in that instance is sufficiently smaller than the minimum list size, the second step consists in choosing simultaneously one colour at random from $L_1(v)$ for each uncoloured vertex $v\in V_1$; this returns a proper $L_1$-colouring of $G[V_1]$, 
and hence a proper $L$-colouring of $G$, 
with positive probability as can be shown using the Lov\'asz Local Lemma. This second step is usually called the Finishing Blow, and follows from a result of Haxell \cite{Hax01}.  
\begin{lemma}[Haxell, 2001] 
Let $H$ be a graph given with a list assignment $L\colon V(H) \to 2^\mathbb{N}$. Suppose that there exists some integer $\ell$ such that $|L(v)| \ge \ell$ for every vertex $v\in V(H)$, and that for every colour $x\in L(v)$, the number of neighbours $u$ of $v$ such that $x\in L(u)$ is at most $\ell/2$. Then there exists a proper $L$-colouring of $H$. 
\end{lemma}

In contrast, our proof does not require the use of the Finishing Blow, and is the first of that nature to the best of our knowledge. 
Our focus is on graphs where the density is bounded within each neighbourhood, a natural extension of triangle-free graphs. The first result in that vein is due to Alon, Krivelevich, and Sudakov \cite{AKS99}. 
They relied on the result of Johansson in order to prove that if a graph $G$ of maximum degree $\Delta$ is such that each neighbourhood spans at most $\Delta^2/f$ edges, where $1\le f \le  \Delta^2+1$, then its chromatic number is at most $O(\Delta/\ln f)$ as $f\to \infty$. This result was then extended to the list chromatic number by Vu \cite{Vu02}.
More recently, it has been proved by the second author together with Davies, Kang, and Sereni \cite[Corrolary 24]{DKPS20+}, that this upper bound can be tightened to $(2+\eps) \Delta/ \ln f$, provided that $f \ge (\ln \Delta)^{2/\eps}$, for every $\eps>0$.
Our main theorem is a strengthening of these results which is close to being optimal.

\begin{thm}\label{thm:main}
Let $G$ be an $n$-vertex graph of maximum degree $\Delta$ such that every graph induced by a neighbourhood in $G$ has average degree at most $d\le \frac{\Delta}{6}-1$. Write $\rho \coloneqq \Delta/(d+1)$ and let $\ell \ge (d+1)(\ln \rho)^3$. Then for every list assignment  $L\colon V(G)\to 2^\mathbb{N}$ with 
\[|L(v)| \ge \pth{1+\frac{2}{\ln \rho}} \frac{\deg(v)}{W\pth{\frac{\deg(v)}{\ell}}}\]
for every vertex $v\in V(G)$,
 there are at least $\ell^n$ proper $L$-colourings of $G$.
\end{thm}

In the statement of Theorem~\ref{thm:main}, we use the $W$-Lambert function $z\mapsto W(z)$ which is defined as the reciprocal of the function $z \mapsto ze^z$. In the proof of Theorem~\ref{thm:main}, we will use the well-known fact that $e^{W(z)} = z/W(z)$. Moreover, we note that $W(z) = \ln z - \ln \ln z + o(1)$ as $z\to \infty$; hence, by fixing $d\coloneqq 0$, Theorem~\ref{thm:main} implies the result of Molloy for triangle-free graphs $G$ of maximum degree $\Delta$, namely $\chi_\ell(G) \le (1+o(1))\Delta/\ln \Delta$. 
In the setting of \cite{Vu02}, Theorem~\ref{thm:main} has the following result as a corollary.

\begin{corollary}
\label{cor:vu}
Let $G$ be a graph of maximum degree $\Delta$, such that every neighbourhood spans at most $\Delta^2/f$ edges, for some $1 \le f \le \Delta^2+1$. Then
$ \chi_\ell(G) \le (1+o(1)) \frac{\Delta}{\ln \min \{\Delta,f\}}$ as $f\to \infty$.
\end{corollary}

\begin{proof}
Let $G$ satisfy the hypothesis of Corollary~\ref{cor:vu}.
It is well known that there exists a $\Delta$-regular graph $H$ and a mapping $\varphi\colon V(H) \to V(G)$ such that $G$ is an induced subgraph of $H$, and for every vertex $v\in V(H)$ the number of edges in $H[N(v)]$ equals that in $G[N(\varphi(v))]$ (see for instance the construction in \cite[Lemma 6]{PiSe21}). So we may assume that $G$ is regular.
The average degree in $G[N(v)]$ is therefore at most $d \coloneqq 2\Delta/f$, for every $v\in V(G)$. We have 
\vspace{-4pt}
 \[ \frac{\Delta}{d+1} = \frac{\Delta f}{f + 2\Delta} \ge \begin{cases}  \frac{f}{3} & \mbox{if $f \le \Delta$,} \\ \frac{\Delta}{3} & \mbox{if $f \ge \Delta$.} \end{cases}
\vspace{-2pt} \]
 
 Setting $\rho\coloneqq \frac{\Delta}{d+1} \ge \min \left\{\frac{f}{3},\frac{\Delta}{3}\right\}$, we let $L\colon V(G)\to 2^\mathbb{N}$ be any list assignment of $G$ with $|L(v)| \ge \pth{1+\frac{2}{\ln \rho}} \frac{\Delta}{W\pth{\frac{\rho}{(\ln \rho)^3}}}$. By Theorem~\ref{thm:main}, $G$ is $L$-colourable. Hence
 \[ \chi_\ell(G) \le \pth{1+\frac{2}{\ln \rho}} \frac{\Delta}{W\pth{\frac{\rho}{(\ln \rho)^3}}} 
\le (1+o(1)) \frac{\Delta}{\ln \min \{\Delta,f\}},\]
 as $f \to \infty$ (and therefore also $\Delta\to \infty$).
\end{proof}

It has been observed in \cite{DJKP20} that there exist graphs satisfying the hypothesis of Corollary~\ref{cor:vu} and with chromatic number $(\frac{1}{2}-o(1))\Delta / \ln \min\{f,\Delta\}$. So the bound in Corollary~\ref{cor:vu} is sharp up to an asymptotic factor $2$. Reducing this gap would constitute a breakthrough in Ramsey theory, as this would imply an improvement of the estimate of the off-diagonal Ramsey numbers $R(3,t)$ for $t\in \mathbb{N}$, a long-standing open problem.

\medskip

In Section~\ref{sec:proof}, we prove Theorem~\ref{thm:main}. Our proof relies on the stronger induction hypothesis that given a uniform random colouring of any induced subgraph $H$ of $G$, the expected number of colours available for a vertex $v\in V(G) \setminus V(H)$ is large enough in expectation.
This is done using only the first moment method.
We were inspired by a counting argument which was first used in the context of graph colouring by Rosenfeld \cite{Ros20} (see \cite[Section 3.5]{Wo20} for an introduction to the method), and later devised more generally for hypergraph colouring by Wanless and Wood \cite{WaWo20+}. We note however that our use of that counting argument does not follow the usual guidelines, as we do not introduce any explicit bad events that our colourings should avoid. 


\medskip

We note that our result also holds in the context of DP-colouring, an extension of list colouring introduced by Dvo\v r\' ak and Postle \cite{DvPo18}. We have decided to state those results in the context of list colouring in order to avoid the verbose formalism of DP-colouring, but our proofs adapt readily to that context.


\section{Proof of the main result}
\label{sec:proof}
Let $G$ be a graph and $X\subset V(G)$.
We write $G[X]$ for the subgraph of $G$ induced by $X$. 
We need the following simple observation.

\begin{lemma}
\label{lem:maxdeg}
Let $H$ be a graph, and $L$ a list assignment of $H$ such that $|L(v)|\ge \deg(v)+1$ for every vertex $v\in V(H)$. 
If $\bc$ is a uniform random proper $L$-colouring of $H$, then given some colour $x\in \bigcup_{v\in V(H)} L(v)$, the probability that $\bc(v)\neq x$ for all $v\in V(H)$ is at least
\[ \prod_{v\in V(H)} \pth{1- \frac{1}{|L(v)|-\deg(v)}}.\]
\end{lemma}

\begin{proof}
Since $H$ is greedily $L$-colourable, we can sample a uniform proper $L$-colouring $\bc$ of $H$. 
Let $v\in V(H)$.
For every $L$-colouring $c'$ of $H'\coloneqq H\setminus v$, one has
\[ \pr{\bc(v)=x \mid \restrict{\bc}{H'}=c'} \le \frac{1}{|L(v)|-\deg(v)},\]
since after removing the colours in $c'(N(v))$ from $L(v)$, there remains at least $|L(v)|-\deg(v)$ possible choices for $\bc(v)$ which are equiprobable. 
Sampling a uniform random $L$-colouring and then resampling the colour of each vertex once gives the result, as the resulting random $L$-colouring is also uniformly distributed. 
\end{proof}

Before proceeding with the proof of Theorem~\ref{thm:main}, we describe the random experiment at the heart of the proof. 
Let $G$ be a graph, $L$ a list assignment of $G$, and $v\in V(G)$. We wish to prove that for $\bc$ drawn uniformly at random from the set of proper $L$-colourings of $G'\coloneqq G\setminus v$, the sublist of $L(v)$ consisting of available colours at $v$ given $\bc$ is large in expectation. 
We do so by analysing the following random procedure.
\begin{enumerate}[(i)]
    \item \label{one} Sample a proper colouring of $G'$ uniformly at random;
    \item \label{two} mark all vertices in $N(v)$ that have short lists \emph{due to the colouring on} $G'\setminus N(v)$;
    \item \label{three} uncolour all unmarked vertices in $N(v)$;
    \item \label{four} choose a proper re-colouring of the uncoloured vertices in $N(v)$ uniformly at random. 
\end{enumerate}
Conveniently, the random proper colouring obtained at the end of this experiment is once again uniformly distributed across all proper colourings of $G'$. 
This is because the lists we mark in step \eqref{two} are selected solely based on the colouring of $G'\setminus N(v)$, which remains fixed after step \eqref{one}.
The proof proceeds by computing a lower bound on the expected size of the list of available colours at $v$ after step \eqref{four}, thus proving the desired induction hypothesis.


\begin{proof}[Proof of Theorem~\ref{thm:main}]
Fix $\rho\coloneqq \frac{\Delta}{d+1} \ge 6$, $t\coloneqq (d+1)(\ln \rho+1)$, and $\ell \ge (d+1)(\ln \rho)^3$.
 For every $v\in V(G)$, let $k(v) \coloneqq \pth{1+\frac{2}{\ln \rho}}\frac{\deg(v)}{W\pth{\frac{\deg(v)}{\ell}}}$. Note that $k(v)\ge \frac{\deg(v)}{W\pth{\frac{\deg(v)}{\ell}}} = \ell e^{W(\deg(v)/\ell)} \ge \ell$.

\smallskip

Let $L$ be any list assignment of $G$ such that $|L(v)| \ge k(v)$ for every vertex $v\in V(G)$.
Given a proper $L$-colouring $c$ of a subgraph $H$ of $G$ and a vertex $v\in V(G)$, we let $L_c(v) \coloneqq L(v) \setminus c(N_H(v))$ be the set of colours still available at $v$ given $c$, and $\ell_c(v) \coloneqq |L_c(v)|$ the number of such colours.
Let $\C(H)$ denote the set of proper $L$-colourings of $H$.
We show by induction that for all induced subgraphs $H\subseteq G$ we have 
\begin{equation}
\label{eq:HI}
\tag{$\star$}
    |\C(H)|\ge \ell \, |\C(H\setminus v)|,
\end{equation}
for all $v\in V(H)$. Observe that, given a colouring $c\in \C(H\setminus v)$, the number of extensions of $c$ to a colouring in $\C(H)$ is precisely $\ell_c(v)$. Hence \eqref{eq:HI} is equivalent to $\esp{\ell_{\bc}(v)}\geq \ell$, for a uniform random colouring $\bc \in \C(H\setminus v)$.

The base case with $H=v$ for some $v\in V(G)$ follows as we have $|\C(\varnothing)|=1$, and $|\C(H)|=k(v)\ge \ell$.
Suppose now that $|V(H)|\ge 2$, and that the induction hypothesis \eqref{eq:HI} holds for all induced subgraphs of $H'\coloneqq H\setminus v$.
We let $\bc$ be drawn uniformly at random from $\C(H')$, and we let $\bc_0$ be obtained from $\bc$ by uncolouring $N_H(v)$; thus $\bc_0$ is a proper $L$-colouring of $H_0\coloneqq H'\setminus N(v)$.
For every $u\in N_H(v)$, we denote by $d_u$ the degree of $u$ within $H[N_H(v)]$, and we let $t_u \coloneqq (d_u+1)(\ln \rho + 1)$. Hence the average of $t_u$ over all $u\in N_H(v)$ is at most $t$.
By the induction hypothesis \eqref{eq:HI}, and using again the observation that the number of extensions of a colouring $c$ to an additional vertex $u$ is $\ell_c(u)$, we know that 
\begin{equation}\label{eq:markov}
      \pr{\ell_\bc(u)\leq t_u} = \frac{|\{c'\in \C(H') : \ell_{c'}(u)\le t_u\}|}{|\C(H')|}\le \frac{t_u \cdot |\C(H'\setminus u)|}{\ell \cdot |\C(H'\setminus u)|} \le \frac{t_u}{\ell}.
  \end{equation} 

Given a colouring $c\in \C(H')$ such that $\restrict{c}{H_0}=c_0$, let $L^0_c(v)$ be obtained by removing from $L(v)$ the colour $c(u)$ for each neighbour $u$ of $v$ such that $\ell_{c_0}(u)\le t_u$.
We let $\bk(v) \coloneqq |L^0_\bc(v)|$. Since there is at most one colour removed for each such neighbour $u$, one has 
\begin{align}
\esp{\bk(v)} &\ge k(v) - \sum_{u\in N(v)}\limits  \pr{\ell_{\bc_0}(u)\le t_u} \ge k(v)-\frac{t \deg(v)}{\ell} \\
& \ge \pth{1+\frac{2}{\ln \rho}}\frac{\deg(v)}{W\pth{\frac{\deg(v)}{\ell}}} - \frac{1}{\ln \rho}\frac{\deg(v)}{\ln \rho - 1 }
\ge \pth{1+\frac{1}{\ln \rho}}\frac{\deg(v)}{W\pth{\frac{\deg(v)}{\ell}}},
\label{eq:esperance(k)}
\end{align}
where we use that $\ln \rho -1 \ge W(\rho/(\ln \rho)^3) \ge W(\deg(v)/\ell)$, since $\rho \ge 6$.
For a colour $x\in L(v)$, we let $\N_c(x)$ be the set of neighbours $u\in N(v)$ such that $x\in L_c(u)$ and $\ell_{c_0}(u) \ge t_u$.
In the final stages of the proof we will need the following inequality, which 
uses that $\frac{\ell_{c_0}(u)}{\ell_{c_0}(u)-d_u-1} \le \frac{t_u}{t_u-d_u-1} = 1+\frac{1}{\ln \rho}$ if $u\in \N_c(x)$ (and so $\ell_{c_0}(u)\ge t_u$). 
That is, we have
\begin{align}
 \sum_{x\in L(v)} \sum_{u \in \N_{c}(x)} \frac{1}{\ell_{c_0}(u)-d_u-1} &\le \sum_{x\in L(v)} \sum_{u \in \N_{c}(x)} \pth{1+\frac{1}{\ln \rho}}\frac{1}{\ell_{c_0}(u)} \\
 &\le \pth{1+\frac{1}{\ln \rho}} \sum_{\substack{u\in N(v) \\ \ell_{c_0}(u) \ge t_u}} \sum_{y\in L_c(u)} \frac{1}{\ell_{c_0}(u)} 
 \le \pth{1+\frac{1}{\ln \rho}}\deg(v). 
 \label{eq:double-sum}
 \end{align}

We will now estimate the expected value of $\ell_{\bc}(v)$.
Let $c_0 \in \C(H_0)$ be a possible realisation of $\restrict{\bc}{H_0}$.
Let $X_0 \coloneqq \{u\in N(v) : \ell_{c_0}(u)\ge t_u\}$ and let $H_1 \coloneqq H'\setminus X_0$, that is $H_1$ includes all vertices with short lists given $c_0$. 
First observe that we have $\ell_{c_1}(u)-\deg_{H[X_0]}(u) \ge \ell_{c_0}(u)-d_u \ge t_u-d_u > 1$ for every vertex $u\in X_0$ and every extension $c_1$ of $c_0$ to $H_1$. 
Indeed,  each vertex in $N(u)\cap N(v)\cap H_1$ can remove at most one colour from $L_{c_0}(v)$.
We now let $c_1$ be a possible realisation of $\restrict{\bc}{H_1}$ given $\bc_0=c_0$. If we assume that $\restrict{\bc}{H_1} = c_1$, then by definition we have $L^0_\bc(v)=L_{c_1}(v)$, and $\N_\bc(x) = \N_{c_1}(x)$ for every colour $x\in L_{c_1}(v)$. By Lemma~\ref{lem:maxdeg} applied on the graph $H[X_0]$ with the list assignment $L_{c_1}$ we have, for every $x\in L_{c_1}(v)$, 
\begin{align}
 &\pr{x\in L_\bc(v) \; \middle|\; \restrict{\bc}{H_1}=c_1} = \pr{\bc(v) \neq x \mbox{, for every } u\in X_0  \; \middle|\; \restrict{\bc}{H_1}=c_1} \\
 &\ge \prod_{u \in \N_{c_1}(x)} \pth{1-\frac{1}{\ell_{c_1}(u) - \deg_{H[X_0]}(u)}} \ge \prod_{u \in \N_{c_1}(x)} \pth{1-\frac{1}{\ell_{c_0}(u) - d_u}} \\
  &= \esp{\prod_{u \in \N_\bc(x)} \pth{1-\frac{1}{\ell_{\bc_0}(u) - d_u}} \; \middle|\; \restrict{\bc}{H_1}=c_1}.\\
\end{align}
We deduce from the law of total expectation that 
\[\esp{\ell_{\bc}(v)} \ge \esp{\sum_{x\in L^0_\bc(v)} \prod_{u\in \N_{\bc}(x)} \pth{1-\frac{1}{\ell_{\bc_0}(u)-d_u}}}.\]

We continue with a sequence of further lower bounds on the right hand side,
where we first use the fact that $1-\frac{1}{x} \ge e^{-1/(x-1)}$ for every $x>1$.
We also use Jensen's inequality together with the convexity of the function $x\mapsto xe^{-a/x}$ for any constant $a\ge 0$, on the domain $(0,+\infty)$.
\begin{align*}
\esp{\ell_{\bc}(v)} 
& \ge \esp{\sum_{x\in L^0_\bc(v)} \exp\pth{\sum_{u\in \N_{\bc}(x)} -\frac{1}{\ell_{\bc_0}(u)-d_u-1}}} \\
&\ge \esp{\bk(v)e^{-\pth{1+1/\ln \rho}\deg(v)/\bk(v)}} & \mbox{by convexity of $\exp$ and \eqref{eq:double-sum},}\\
&\ge \esp{\bk(v)}e^{-\pth{1+1/\ln \rho}\deg(v)/\esp{\bk(v)}} & \mbox{by Jensen's inequality,} \\
&\ge \pth{1+\frac{1}{\ln \rho}}\frac{\deg(v)}{W\pth{\frac{\deg(v)}{\ell}}} e^{-W\pth{\frac{\deg(v)}{\ell}}} \ge \ell &\mbox{by \eqref{eq:esperance(k)}.}
\end{align*}

This ends the proof of the induction.
\end{proof}

In the last equation of the above proof, it is interesting to note that the transition from the first to the second line informs us that the worst case for our lower bound occurs when all colours $x\in L^0_\bc(v)$ 
are equally likely to appear in $L_\bc(v)$.
Moreover, the transition from the second to the third line shows that the worst case is when the appearance of short lists in $N(v)$ is evenly distributed across all colourings $c_0\in \C(H_0)$.

%
%

\section{A precise count of the number of colourings}
Shortly before this manuscript was published in a preprint repository, a similar independent work by Bernshteyn, Brazelton, Cao, and Kang appeared online \cite{BBRK21+}. They also rely on the counting argument used by Rosenfeld in \cite{Ros20}, and prove the following.

\begin{thm}[Bernshteyn, Brazelton, Cao, Kang; 2021]
\label{thm:Anton}
For every $\eps>0$, there exists $\Delta_0$ such that the following holds. Let $G$ be an $n$-vertex triangle-free graph of maximum degree $\Delta\ge \Delta_0$ and with $m$ edges. Then, for every $q\ge (1+\eps) \Delta/\ln \Delta$, the number of proper $q$-colourings of $G$ is at least
$ (1-1/q)^m \big((1-\delta)q \big)^n$,
where $\delta=\frac{4}{q} e^{\Delta/q}$.
\end{thm}

{They in fact go further and show that the bound on the number of colourings given by Theorem~\ref{thm:Anton} is asymptotically sharp for $\Delta$-regular triangle-free graphs.}
In particular, they prove that a random $\Delta$-regular triangle-free graph will almost surely admit no more than $(1-\frac{1}{q})^{\Delta n/2} \pth{\pth{1+\frac{2\ln n}{n}}q }^n$ $q$-colourings as $n \to \infty$.

When $q=(1+o(1))\Delta/\ln \Delta$, the number of $q$-colourings given by Theorem~\ref{thm:Anton} for every triangle-free graph $G$ is $e^{\pth{\frac{1}{2}-o(1)}n\ln \Delta}$.
We note that Theorem~\ref{thm:main} applied with $\ell=\sqrt{\Delta}$ implies the same bound on the number of $2q$-colourings, so with twice as many colours. 
We now argue that we can slightly adapt the proof of Theorem~\ref{thm:main} to give a bound that matches the number of colourings given by Theorem~\ref{thm:Anton} and that applies in the more general setting of Theorem~\ref{thm:main}.

\begin{thm}\label{thm:main-tight}
Let $G$ be an $n$-vertex graph of maximum degree $\Delta$ such that every graph induced by a neighbourhood in $G$ has average degree at most $d \le \frac{\Delta}{6}-1$. Let $f \coloneqq \Delta/(d+1)$ and
suppose $L\colon V(G)\to 2^\mathbb{N}$ is a list assignment  with $|L(v)| \ge \pth{1+\frac{1}{\ln \rho}} q(v)$, where  
\[q(v) \ge \pth{1+\frac{1}{\ln \rho}} \frac{\deg(v)}{W\pth{\frac{\deg(v)}{(d+1)(\ln \rho)^3}}}\]
 for every vertex $v\in V(G)$.
Then there are at least $(q\big/\sqrt{D})^n$
 proper $L$-colourings of $G$, where $D$ is the geometric mean of the degrees in $G$, and $q$ is the geometric mean of $\{q(v)\}_{v\in V(G)}$.
\end{thm}
\begin{proof}
Following the same set-up as in the proof of Theorem~\ref{thm:main} we have 
$\esp{\bk(v)} \ge q(v)$, for every vertex $v\in V(H)$.
We also note that in \eqref{eq:double-sum}, we can replace $\deg(v)$ with $\deg_H(v)$. 
Then
\begin{align}
\label{eq:tight-expectancy}
\esp{\ell_\bc(v)} &\ge \esp{\bk(v)}e^{-\pth{1+\frac{1}{\ln \rho}}\deg_H(v)/\esp{\bk(v)}} \ge q(v)e^{-\pth{1+\frac{1}{\ln \rho}}\deg_H(v)/q(v)}.
\end{align}

We now let $v_1,\ldots,v_n$ be an ordering of $V(G)$ such that $\pth{q(v_i)}_{i=1}^n$ is non-decreasing. Letting $H_1$ be the empty graph, and $H_i \coloneqq G[v_1, \ldots, v_{i-1}]$ for every $2\le i \le n$, we apply \eqref{eq:tight-expectancy} on the pairs $(H_i,v_i)$ for every $1\le i \le n$ and obtain that the number of $L$-colourings of $G$ is
\begin{align*}
|\C(G)| &\ge \prod_{i=1}^n q(v_i) e^{-\pth{1+\frac{1}{\ln \rho}}\deg_{H_i}(v_i)/q(v_i)} \\
& = q^n \exp \pth{-\pth{1+\tfrac{1}{\ln \rho}} \sum_{uv \in E(G)}\limits \min \left\{ \frac{1}{q(u)}, \frac{1}{q(v)} \right\}} \\
&\ge q^n \exp \pth{-\pth{1+\tfrac{1}{\ln \rho}} \sum_{uv \in E(G)}\limits \pth{ \frac{1}{2q(u)} + \frac{1}{2q(v)} }} \\
&= q^n  \exp\pth{-\pth{1+\tfrac{1}{\ln \rho}} \sum_{i=1}^n \limits\frac{\deg(v_i)}{2q(v_i)}} \\
& \ge  q^n  \exp\pth{-\frac{1}{2} \sum_{i=1}^n \limits \ln \deg(v_i)} = q^n D^{-n/2}.
\qedhere
\end{align*}
\end{proof}


\section{Discussion}

\subsection{Simplicity of the proof}

A classical intuition behind the result of Johansson is the following. Consider any vertex $v\in V(G)$ in a triangle-free graph G, and assume that we properly colour its neighbourhood randomly. 
Since $N(v)$ is empty, the colours drawn by each vertex in $N(v)$ are independent random variables.
The solution to the Coupon Collector Problem hence tells us that if $v$ has $k$ colours (initially) available and degree less than $\Omega(k\ln k)$, then with high probability some colours from $L(v)$ do not appear in $N(v)$. 
This shows that a random colouring ``works" locally.
We believe that our proof closely approximates this intuition. 
Following Molloy \cite{Mol19}, both Bernshteyn \cite{Ber19} and ourselves have relied on the approximation of $(1-1/t)$ by $e^{-\frac{1}{t-1}}$, where $t$ is the order of some list in $N(v)$ given a colouring of the vertices outside $N(v)$. 
This allows us to fruitfully exploit the convexity of the latter and assume that list sizes are evenly distributed across the colourings from which we sample. 
However, this approximation is only possible if $t>1$ and is wasteful if $t$ is small. 
Thus previous proofs introduced a blank colour (which lets them replace $t$ with $t+1$) and showed that however an adversary coloured $G\setminus N(v)$, a random colouring of $N(v)$ left many colours available at $v$ with high probability.
In particular this required showing that, having fixed any colouring of $G\setminus N(v)$, the appearance of colours $x$ and $y$ in $N(v)$ were negatively correlated.
We instead require $t$ to be large enough to ensure the convex approximation is good and we banish the adversary by sampling uniformly from the set of colourings of $G\setminus v$ (which typically induce large lists by induction and \eqref{eq:markov}).
Thus the ``support" of the convex approximation, and its power, is enlarged to all colourings of $G \setminus v$ rather than than just the extensions of a particular colouring $c_0$ of $G\setminus N(v)$.   
Of course, sampling from all colourings simultaneously  comes at a cost... the appearance of colours in $N(v)$ becomes tangled in a complex knot of correlations. 
Crucially, we do not attempt to untangle this knot and instead  slice through it by focusing only on the first moment. 
We believe that our technique could be adapted in other contexts where complex correlations obstruct intuitive arguments.
Can one systemically replace some family of R{\"o}dl nibble arguments using our technique? An affirmative answer would be extremely exciting.

\subsection{A possible tightening of the bound}
In a previous version of our work, we have claimed that the events $x\in L_\bc(v)$ for every colour $x\in L(v)$, given a vertex $v\in V(H)$ and a uniform random colouring $\bc \in \C(H)$, are negatively correlated. As was pointed out to us by Anton Bernshteyn, this claim is false, although it holds once a colouring of $H \setminus N[v]$ is fixed.
If one could prove that $\ell_\bc(v)$ is highly concentrated around its mean, then this would imply a tightening of the bound in Theorem~\ref{thm:main}. However, it is not clear whether such a high concentration holds for a uniform random colouring $\bc \in \C(H)$.

\section{Acknowledgement}
A substantial part of this work has been done during the online workshop \emph{Entropy Compression and Related Methods} which took place in March 2021. We are thankful to the organisers, Ross J. Kang and Jean-Sébastien Sereni, and more generally to the Sparse Graph Coalition for making that work possible. 
We are grateful to Felix Joos for proofreading the early versions of that paper.
We also thank Matthieu Rosenfeld and Mike Molloy for insightful discussions.
\bibliography{moment}
\bibliographystyle{habbrv}

\end{document}